\documentclass{amsart}
\usepackage{graphicx}
\usepackage{amscd}
\usepackage{amsmath}
\usepackage{amsfonts}
\usepackage{amssymb}
\theoremstyle{plain}
\newtheorem{thm}{Theorem}[section]
\newtheorem{cor}[thm]{{Corollary}}

\newtheorem{lem}[thm]{{Lemma}}
\newtheorem{prop}[thm]{Proposition}
\theoremstyle{definition}

\newtheorem{example}[thm]{Example}
\theoremstyle{remark}

\title[On analytic functions in an ordered field...]{On analytic functions in an ordered field with an infinite rank valuation}
\author{H\'ector M. Moreno}
\address{Departamento de Matem\'aticas, Universidad de La Serena, Avenida Cisternas $1200$, La Serena, Chile}
\email{hmoreno@userena.cl}
\thanks{Supported by Proyecto PR18152 Direcci\'on de Investigaci\'on y Desarrollo, Universidad de La Serena}

\begin{document}

\begin{abstract}

\noindent Let $K$ be the scalar field of the first orthomodular (or Form Hilbert) space, described by H. Keller in $1980$. It has a non-Archimedean order, an infinite rank valuation compatible with the order as well as an explicitly defined ultrametric, all of which induce the same topology. 

\vspace{1ex}

\noindent We study analytic functions defined on valued field $K$, and we will establish an invertibility local theorem for these functions as an application of Banach fixed point theorem on a particular
 complete metric space.
\end{abstract}

\maketitle

\section{Preliminaries}

\noindent Let $K$ be the scalar field of the first orthomodular space, described by H. Keller (\cite{Keller}), which is provided with an infinite rank valuation. H. Moreno introduced in \cite{Hector1}, \cite{Hector3} and \cite{Hector2} an Ultrametric Calculus on $K$, where its studied some properties of analytic functions and strictly continuous differentiable functions on $K$. In this article, we will continue the study of analytic functions defined on $K$, using the fact that the order topology is induced by an infinite rank valuation
as well an ultrametric. Moreover, following ideas of \cite{Hector2}, its posible to prove that analytic functions with non-null derivative in an open set has a local analityc inverse.

\vspace{1ex}

 \noindent  We shortly review the construction of the field $K$ described in detail in \cite{Hector1} and \cite{Keller}. Let $F_0:=\mathbb{R}$ and $\{X_1,X_2,X_3,\ldots\}$  a set of variables. For each $n\geq 1$, we define $F_n=F_{n-1}(X_n)=\mathbb{R}(X_1,\dots, X_n)$ and $F_{\infty}=\displaystyle \bigcup_{n\in\mathbb{N}}F_n$.
 
 
 \vspace{1ex}
 
\noindent Considering $\mathbb{R}$ with the usual order, we define the order on $F_{\infty}$ inductively as follows: 
$P(X_n)=a_0+ \ldots + a_s X_n^s\in F_{n-1}[X_n]$ is positive if and only if $a_s$ is positive. On the another hand, 
 a quotient of polynomials $\displaystyle{\frac{p(X_n)}{q(X_n)}}\in F_n$,  with $q(X_n) \neq 0$, is positive if and only if
$p(X_n)q(X_n)$ is positive. $(F_\infty,\leq )$ is a non-Archimedean ordered field,
 since the variable $X_1$ is an upper bound of $\mathbb{N}$ . 

\vspace{1ex}

\noindent  We define the field $K$ as the completion of $F_{\infty}$ by Cauchy sequences with respect to the order topology.  The order of $F_{\infty}$ is extended to $K$ by usual arguments, and
we have that $K$ is an ordered field. The order induces an  absolute value $|\,|$, in the classical sense,   $|x|=\max\{x,-x\}$ for all $x$.

\vspace{1ex}

\noindent \noindent Now we will introduce a non-archimedean valuation on $K$. The interplay between the order and this valuation will be crucial in the proof of the results presented in this article.
The value group of the valuation of $K$ is 
$$\Gamma:=\left\{\gamma\in (g_1^{n_1},g_2^{n_2},g_3^{n_3},...,g_i^{n_i},...)\in\prod_{i=1}^{\infty}G_i:\;\text{$n_i\in\mathbb{Z}$  such that $supp(\gamma)$ is finite}\right\},$$
where $supp(\gamma)=\{i\in\mathbb{N}:\, n_i\neq 0\}$, and  each $G_i$ is a multiplicatively subgroup generated
by a real number $g_i>1$ ordered by the usual ordering of $\mathbb{R}$. $\Gamma$ is a
linearly ordered group with the componentwise operation and the
antilexicographical ordering, that is,  if $0\neq (g_j)_{j\in\mathbb{\mathbb{N}}}\in\Gamma$ and $m=\max supp ((g_j))$, then
$$(g_j)_{j\in\mathbb{N}} >0\,\,\text{in}\,\, \Gamma\Longleftrightarrow g_m >0\,\,\text{in}\,\, G_m.$$

\vspace{1ex}

\noindent The non-archimedean valuation $v_0:\, K\rightarrow \Gamma\cup\{0\}$  is trivial on $\mathbb{R}$ and maps each $X_i$ to $\hat{g}_n :=(1,...,1,g_n,1,...)\in\Gamma$.
 $v_0$ can be extended uniquely to a non-archimedean valuation $v$ on K with the same value group. 
 
\vspace{1ex}
 
\noindent The valuation $v$ is {\it compatible} with the order defined before on $K$ in the following sense:

\centerline{\it For all  $a,b\in K$  if  $|a|\leq |b|$  then  $v(a)\leq v(b)$.}
\vspace{1ex}

\noindent This implies that the topologies induced by the order and the valuation are identical on $K$. 
 Moreover, this common topology $\tau$ on $K$ is (ultra)metrizable and  the ultrametric on $K$ is defined by the map
$d(x,y)=\phi(|x-y|)$ where
$\phi(0)=0$, and
$$\phi(x)=2^{-\min\{m\in\mathbb{N}:\; X_m^{-1}\leq x\}}\qquad (x>0).$$

\vspace{1ex}

\noindent From the definitions of $\leq$ and $|\,|$ on $K$, we can conclude that the following inclusions are hold
$$\{x\in K:\, |x-a|<\frac{1}{X_n}\}\subset \{x\in K:\; v(x-a)< \hat{g}_{n}^{-1}\}\subset \{x\in K:\, |x-a|<\frac{1}{X_{n-1}}\}$$ 
for each $a\in K$ and $n\geq 1$. For $r\in\Gamma$ and $a\in K$, we define the {\it open ball
with center $a$ and radius $r$} is the set
$$B_a(r^-)=\{x\in K:\; v(x-a)<r\}.$$\

\noindent Respectively, the {\it closed ball
with center $a$ and radius $r$} is the set
$$B_a(r^-)=\{x\in K:\; v(x-a)\leq r\}.$$

\vspace{1ex}

\section{Analytic functions.}


\noindent Given a sequence $a,a_0,a_1,...$ in $K$,  we define a power series $\displaystyle\sum_{j=0}^{\infty}a_j(z-a)^j$ in the classical sense.
 Since the valued group $\Gamma$ has an infinitude of convex subgroups (see \cite{Hector1}) and the order on $K$ is non-Archimedean,  if $x\in K$ then exists $y\in K$ such that $x^n<y$ for all $n\in\mathbb{N}$. Then, we obtain the following theorem proved in 
\cite{Hector1} and \cite{Hector3} concerning convergence of power series. 

\begin{thm}[\cite{Hector1},\cite{Hector3}]\label{convergencia series}
 Let $a_0,a_1,...$ in $K$. The power series $\displaystyle{\sum_{j=0}^{\infty}a_jz^j}$ converges, for $z \neq 0$,
  if and only if $\displaystyle\lim_{n\rightarrow\infty}a_n=0$.
  \end{thm}

\noindent Let $D\subseteq K$ be a non-empty open set. We shall say that a function $f:D\rightarrow K$ is analytic in $D$ if there exist $u\in D$, $r\in\Gamma$ and  $a_0,a_1,\ldots\in K$ such that for every $z\in D$
$$f(z)=\sum_{n=0}^{\infty}a_n(z-u)^n.$$

\noindent An analytic function is infinitely many times differentiable in the order topology. Considering $f(z)$ as above we have $\displaystyle f'(z)=\sum_{n=1}^{\infty}n\, a_n(z-u)^{n-1}$ for $z\in D$. The 
following theorem  says that it does not matter the choose of $u\in D$ in the definition of analyticity. 

\begin{thm}[\cite{Hector1},\cite{Wim}]\label{infinitos centros}
Let $f$ be an analytic function in a open subset $D$. Then for every  $v\in D$ there exists $b_0,b_1,\ldots\in K$ such that $\displaystyle f(z)=\sum_{n=0}^{\infty}b_n(z-v)^n$
for all $z\in D$.
\end{thm}

\noindent Although the topology $\tau$ defined on $K$ is zero dimensional, the set of zeros of an analytic function on a ball $B_a(r)$ does not have an accumulation point (see \cite{Hector1} for details). Then, applying the theorems \ref{convergencia series} and \ref{infinitos centros}, it is possible to show that  a function $f(z)$ is analytic in $D$ if and only if there exists $a_0,a_1,...\in K$ such that for every $z\in D$
$$f(z)=\sum_{n=0}^{\infty}a_nz^n.$$
 
\noindent To prove that a composition of analytic functions is an analytic function is necessarily introduce some results of convergence of series in $K$. 

\vspace{1ex}

\noindent The following proposition was proved in the case of Levy-Civita field, but the proof  is valid in the case of $K$.
 
\begin{prop}\label{reordenacion de series 2}
Let  $\{a_{mn}:\; m,n\in\mathbb{N}\}$  be a subset of  $K$ such that $\displaystyle\lim_{n\rightarrow\infty}a_{mn}=0$ for every  $m\in\mathbb{N}$ and $\displaystyle\lim_{m\rightarrow\infty}a_{mn}=0$ uniformly in $n$. Then,
$$\sum_{n=0}^{\infty}\sum_{m=0}^{\infty}a_{mn}=\sum_{m=0}^{\infty}\sum_{n=0}^{\infty}a_{mn}.$$
\end{prop} 
 
\begin{lem}
Let $\displaystyle\sum_{n=0}^{\infty}a_nz^n$ and $\displaystyle\sum_{n=0}^{\infty}b_nz^n$ power series on $K$. Suppose that both converge on $K$, then the function
$$h(z)=\displaystyle\sum_{n=0}^{\infty}a_n\left(\sum_{k=0}^{\infty}b_kz^k\right)^n$$
is well defined on $K$, and $h(z)=\displaystyle\sum_{m=0}^{\infty}c_mz^m$ where
$$c_m=\sum_{n=0}^{\infty}a_n b_m^{(n)},$$
$b_0^{(0)}=1$, $b_m^{(0)}=0$ for $m>1$, $\displaystyle b_m^{(n)}=\sum_{j_1+j_2+\ldots+j_n=m}b_{j_1}b_{j_2}\ldots b_{j_n}$ $(n\in\mathbb{N})$.
\end{lem} 

\begin{proof}
By theorem \ref{convergencia series}, $h(z)$ is well defined since $\displaystyle\lim_{n\rightarrow\infty}a_n=0$.  By induction, it can be prove directly that for every $n\in\mathbb{N}$
$$\left(\sum_{k=0}^{\infty}b_kz^k\right)^n=\sum_{m=0}^{\infty}b_m^{(n)}z^m$$
where $b_m^{(n)}$ are as above. Hence, $\displaystyle h(z)=\sum_{n=0}^{\infty}a_n\left(\sum_{m=0}^{\infty}b_m^{(n)}z^m\right)$.

\noindent We fix $z_0\in K$ arbitrary, then there exists $y\in K$ such that $(z_0)^n\leq y$ for all $n\in\mathbb{N}$. Let $c_{mn}=a_n b_m^{(n)}z_0^m$. Since
$\displaystyle\sum_{n=0}^{\infty}a_nz^n$ and $\displaystyle\sum_{n=0}^{\infty}b_nz^n$ converge on $K$, by theorem \ref{convergencia series} there exist $a, b\in K$ such that
$v(a_n)\leq v(a)$ and $v(b_n)\leq v(b)$ for all $n\in\mathbb{N}$ respectively. Therefore, by the definition of $b_m^{(n)}$
$$v(c_{mn})=v(a_n b_m^{(n)}z_0^m)\leq v(a_n)v(b)^nv(y)\leq v(a_n)v(c)v(y)$$
for some $c\in K$. We obtain that $\displaystyle\lim_{n\rightarrow\infty}c_{mn}=0$ uniformly on $m$. 

\vspace{1ex}

\noindent On the another hand, 
$$v(c_{mn})=v(a_n b_m^{(n)}z_0^m)\leq v(b_m^{(n)}) v(a)v(y).$$
But the power series $\displaystyle \sum_{m=0}^{\infty}b_m^{(n)}z^m$ converges for each $n$, and it follows that 
$$\displaystyle \lim_{m\rightarrow\infty}c_{mn}=\lim_{m\rightarrow\infty}b_m^{(n)}=0$$
for each $n\in\mathbb{N}$. Therefore, by the Proposition \ref{reordenacion de series 2}
$$h(z)=\sum_{n=0}^{\infty}\sum_{m=0}^{\infty}c_{mn}=\sum_{m=0}^{\infty}\sum_{n=0}^{\infty}c_{mn}=\sum_{m=0}^{\infty}\left(\sum_{n=0}^{\infty}a_nb_m^{(n)}\right)z^m.$$
\end{proof}
 
 \noindent A direct consequence of the previous lemma is the following result.
 
 \begin{prop}\label{composicion de analiticas}
 Let $f$, $g$ analytic functions on open sets $D$ and $C$ respectively. If $f(D)\subset C$ then $g\circ f$ is analytic on $D$.
 
 \end{prop}
 
\vspace{1ex}

\noindent  For the proof of the next theorems we need to introduce the residual field associated to $K$. Let us consider now the following sets
$$R=\{ z \in K : v(z) \leq 1\}=B_0(1)$$
$$D=\{ z \in K : v(z) < 1\}=B_0(1^-)$$
Then  $R$ is a local ring with maximal ideal $D$. The residual field is $\hat{k}:=R/D$,
it is $\hat{k}$ is isomorphic  to $\mathbb{R}$, therefore is an Archimedean ordered field, so 
there exists an order-embedding $\phi:\, \hat{k}\rightarrow \mathbb{R}$. The canonical 
homomorphism from $R$ to $\hat{k}$ is the map $x\mapsto\pi(x)$.

\begin{thm}[Maximum principle]\label{modulo maximo}
 Let $f$ be an analytic function on $B_0(r)$ given by $\displaystyle\sum_{n=0}^{\infty}a_nz^n$, then there exists $\max\{v(f(z)):\;v(z)\leq r\}$ in $\Gamma$ and
$$\max\{v(f(z)):\;v(z)\leq r\}=\max\{v(f(z)):\;v(z)=r\}=\max_{n\in\mathbb{N}}\{v(a_n)r^n\}<\infty.$$
\end{thm}

\begin{proof}
  Let $f$ be an analytic function on $B_0(1)$, $\hat{k}=B_0(1)/B_0(1^-)$ the residual field and $x\mapsto\pi(x)$ the quotient map $B_0(1)\rightarrow\hat{k}$. Without loss of generality
we suppose that $\displaystyle \max_{n}v(a_n)=1$.

\vspace{1ex}

\noindent  Since $\displaystyle\lim_{n\rightarrow\infty}a_n z^n=0$ we have $\pi(a_n)=0$ for $n$ large, so the quotient map induces a nonzero polynomial
$$(\pi(f))(x)=\sum_{n=0}^{m}\pi(a_n)x^n\in \hat{k}[X]$$
for some $m\in\mathbb{N}$. $\pi(f)$ has only finite many zeros and $|\hat{k}|=\infty$, so there exists $s\in \hat{k}$ such that $s\neq \pi(0)$ and $(\pi(f))(s)\neq\pi(0)$ .
Let $b\in K$ with $\pi(b)=s$, so we have
$$v\left(\sum_{n=0}^{\infty}a_nb^n\right)=1$$
and hence
$$\max\{v(f(z)):\;v(z)\leq 1\}=\max\{v(f(z)):\;v(z) = 1\}=1\left(=\max_{n\in\mathbb{N}}\{v(a_n)(1)^n\}\right).$$
To prove the case $r\in\Gamma$ arbitrary, we can repeat the arguments above to the function $g(z)=f(a z)$ with $a\in K$, $v(a)=r$ and $z\in B_0(1)$.
\end{proof}

\begin{cor}[Liouville theorem]
A bounded analytic function $f:\, K\rightarrow K$  is constant.
\end{cor}

\begin{proof}
 By theorem \ref{infinitos centros}, there exist $a_0,a_1,a_2,\ldots \in K$ that $f(z)=\displaystyle\sum_{n=0}^{\infty}a_nz^n$ for all $z\in K$. Applying theorem \ref{modulo maximo}, we have that $v(a_n)s^n\leq r$ for all $s\in\Gamma$,
 which implies that $v(a_n)\leq (s^n)^{-1}r$.
In particular, considering the sequence $(s_m)$ defined by $s_m=X_m$, we observe 
$$v(a_n)\leq\lim_{m\rightarrow\infty}v(X_m^{n})^{-1}r=0$$
for all $n\geq 1$. Then, $a_n=0$ for all $n\geq 1$.
\end{proof}

\noindent Let $f:D\rightarrow K$ an non-constant analytic function and let $x_0\in K$. By theorem \ref{infinitos centros}, $f$ has a power expansion series about $x_0$ of the form
$$f(z)=\sum_{n=0}^{\infty}a_n(z-x_0)^n\quad\quad (z\in K).$$
By usual arguments we prove that $\displaystyle a_n=\frac{f^{(n)}(x_0)}{n!}$ for each $n\in\mathbb{N}$. Therefore, for all $x\in K$ there exist 
$m_x\in \mathbb{N}\setminus\{0\}$ such that $f^{(m_x)}(x)\neq 0$, otherwise $f(z)$ would be a constant function. Then, for each $x\in K$ 
$$\min\{n\in \mathbb{N}\setminus\{0\}:\, f^{(n)}(x)\neq 0\}$$
exists.

\begin{prop}\label{signo derivada}
Let $f:D\rightarrow K$ a non-constant analytic function, and let $x_0\in D$. If 
$$m=\min\{n\in \mathbb{N}\setminus\{0\}:\, f^{(n)}(x_0)\neq 0\},$$
then $x_0$ is a relative extremum if and only if $m$ is even. In this case, $x_0$ is a maximum if $f^{(m)}(x_0)<0$, and is a minimum if $f^{(m)}(x_0)>0$.
\end{prop}

\begin{proof}

If $f$ is an analytic function on $D$, then $f$ has a power series expansion of the form
$$f(z)=\sum_{n=0}^{\infty}\frac{f^{(n)}(x_0)}{n!}(z-x_0)^n\qquad(z\in D).$$
Let $h\in K$, we replace $z$ by $x_0+h$ in the expression above and we obtained
$$f(x_0+h)-f(x_0)=\frac{f^{(m)}(x_0)h^m}{m!}+\sum_{n=m+1}^{\infty}\frac{f^{(n)}(x_0)}{n!}h^n.$$
For the first part, we prove that there exist a neighbourhood $V$ of $x_0$ such that for each $h\in V$,
 $f(x_0+h)-f(x_0)$ and $\displaystyle\frac{f^{(m)}(x_0)h^m}{m!}$ have the same sign.

\vspace{1ex}

\noindent The power series $\displaystyle\sum_{n=0}^{\infty}\frac{f^{(n)}(x_0)}{n!}(z-x_0)^n$ converges in $K$, so $\displaystyle\lim_{n\rightarrow\infty}\frac{f^{(n)}(x_0)}{n!}=0$. 
Then, by theorem \ref{convergencia series} there exist $g\in G$ such that for all $n\geq m$
$$v\left(\frac{f^{(n)}(x_0)}{n!}\right)<g,$$
as well as $g_1,g_2\in G$ such that 
$${g}_{1}^{-1}< g\qquad\text{y}\qquad\displaystyle {g}_{2}^{-1}<g^{-1}v\left(\frac{f^{(m)}(x_0)}{m!}\right).$$
 \noindent Let $\delta\in K$ with $v(\delta)<\min\{g_1^{-1},g_2^{-1},1\}$, then for each $h\in(-\delta,0)\cup(0,\delta)$ and $n>m$ we have
 $$v(h)\leq v(\delta)< \min\{{g}_{1}^{-1},{g}_{2}^{-1},1\}\qquad\text{and}$$
 $$v\left(\frac{f^{(n)}(x_0)}{n!}h^{n-m}\right)< g\min\{{g}_{1}^{-1},{g}_{2}^{-1},1\}\leq v\left(\frac{f^{(m)}(x_0)}{m!}\right),$$
 that is,
$$\displaystyle v\left(\frac{f^{(n)}(x_0)h^{n}}{n!}\right)<v\left(\frac{f^{(m)}(x_0)h^m}{m!}\right).$$
\noindent Therefore, 
$$v\left(\sum_{n=m+1}^{\infty}\frac{f^{(n)}(x_0)}{n!}h^n\right)\leq \max\left\{v\left(\frac{f^{(n)}(x_0)}{n!}h^n\right):\,n\geq m+1\right\}<v\left(\frac{f^{(m)}(x_0)h^m}{m!}\right),$$
and which implies 
$$\left|\sum_{n=m+1}^{\infty}\frac{f^{(n)}(x_0)}{n!}h^n\right|<\frac{|f^{(m)}(x_0)||h|^m}{m!}.$$
\noindent By properties of the absolute value, from the expression 
$$f(x_0+h)-f(x_0)=\frac{f^{(m)}(x_0)h^m}{m!}+\sum_{n=m+1}^{\infty}\frac{f^{(n)}(x_0)}{n!}h^n,$$
we conclude
 $$f(x_0+h)-f(x_0)\quad \text{and} \quad\displaystyle\frac{f^{(m)}(x_0)}{m!}h^{m}$$
  have the same sign for all $h\in (-\delta,0)\cup (0,\delta)$.
 
 \vspace{1ex}
 
\noindent Now, $x_0$ is a relative extremum of $f$ if and only if the sign of $f(x_0+h)-f(x_0)$ is the same for each $h\in (-\delta,0)\cup (0,\delta)$. This is true if and only if $m$ is even, otherwise $h^m$ is negative if $h<0$ and positive if $h>0$. Then, $x_0$ is relative maximum (resp. minimum) if and only if $m$ is even and  $f^{(m)}(x_0)<0$ (resp. $f^{(m)}(x_0)>0$). 
 
\end{proof}

\noindent A direct consecuence of the previous proposition is the following corollary.

\begin{cor}
If $x_0$ is a relative extremum of $f$, then $f'(x_0)=0$.
\end{cor}

\begin{prop}
Let $f:[a,b]\rightarrow K$ an analytic function. If  $f'(z)>0$ for all $z\in [a,b]$ then $f(a)<f(b)$.
\end{prop}

\begin{proof}
\noindent Firstly, we suppose that $f$ analytic on $[0,1]$ and $f(0)=0$. Let $\displaystyle f(z)=\sum_{j=0}^{\infty}a_jz^j$ for  $z\in [0,1]$, without loss of generality 
we can assume that  $\displaystyle\max_jv(a_j)=1$. As in the proof of theorem \ref{modulo maximo}, the quotient map $x\mapsto \pi(x)$ induces a nonzero polynomial
$$(\pi(f))(x)=\sum_{n=0}^{m}\pi(a_n)x^n\in \hat{k}[X].$$

\noindent We remind that the residual field is $\hat{k}:=R/D$ is isomorphic  to $\mathbb{R}$ by an order-embedding $\phi:\,\hat{k}\rightarrow \mathbb{R}$. Let $p(x)\in\mathbb{R}[x]$ the polynomial resulting from $f$ under $\phi\,\circ\,\pi$ given by
$\displaystyle p(x)=\sum_{n=0}^{m}(\phi\,\circ\, \pi)({a_n})x^n$. Then, the derivative $p'(x)$ is the polynomial resulting from $f'$ under $\phi\circ \pi$.

\vspace{1ex}

\noindent  For every $\alpha\in (0,1)$ in $K$ we have that $f'(\alpha)>0$, then 
$f'(\alpha)=\pi(f'(\alpha))+\delta_{f'(\alpha)}$ for some $\delta_{f'(\alpha)}\in K$ with $v(\delta_{f'(\alpha)})<1$.  Hence, we can conclude that $p'(x)\geq 0$ in $\mathbb{R}$ for all $x\in\mathbb{R}$.

\vspace{1ex}

\noindent Since $\displaystyle\max_jv(a_j)=1$, $p(x)$ is a non-null polynomial and its derivative is positive on $[0,1]$ in $\mathbb{R}$, and $p(0)$=0. Therefore $p(0)<p(1)$, which implies that
$$f(1)=\pi(f(1))+\delta_{f(1)}$$
with $\pi(f(1))\neq \pi(0)$ and $v(\delta_{f(1)})<1$. By the definition of the order of $K$, $\hat{k}$ and $\phi$, we have that $f(0)<f(1)$. 

\vspace{1ex}

\noindent Suppose now that $f$ is analytic on $[a,b]$ and $f'\geq 0$ on $[a,b]$, we can repeat the argument above before to the function 
$F (x) = f ((b-a)x+a)-f(0)$ which is an analytic function on $[0,1]$, $F(0)=0$ and its derivative is positive on $[0,1]$.
\end{proof}

\begin{cor}
Let $f:[a,b]\rightarrow K$ an analytic function. 
\begin{enumerate}
\item If $f'(z)>0$ for all $z\in (a,b)$ then $f$ is monotone increasing in $[a,b]$. 
\item If $f'(z)<0$ for all $z\in (a,b)$ then $f$ is monotone decreasing in $[a,b]$. 
\end{enumerate}
\end{cor}	
			
\noindent The following example shows the existence of an analytic function $f$ on $[a,b]$ with $f'(z)\neq 0$ on $[a,b]$, but it does not have relatives extremes on $[a,b]$.

\begin{example}
We consider $f(z):\,[1,X_1^2]\rightarrow K$ an analytic function defined by $\displaystyle f(z)={1\over 3} z^3-X_1z$. The derivative of $f(z)$ is $f'(z)=z^2-X_1$, and we observe that
 $f'(z)\neq 0$ for all $z\in [1,X_1^2]$. 

\vspace{2ex}

\noindent Let $z_1,z_2\in K$. Consider the following cases:

\begin{enumerate}
 \item Suppose that $X_1\leq z_1\leq z_2$, then
 $$f(z_1)=z_1\left({{z_1^2}\over 3}-X_1\right)\leq z_1\left({{z_2^2}\over 3}-X_1\right)\leq z_2\left({{z_2^2}\over 3}-X_1\right) =f(z_2).$$

\item If $s<z_1\leq z_2< X_1$ for all $s\in\mathbb{R}$, it can be prove directly that 
$$\displaystyle \min\left\{\left({{z_1^2}\over 3}-X_1\right), \left({{z_2^2}\over 3}-X_1\right)\right\}\geq 0.$$
As the previous part, it is proven that $f(z_1)\leq f(z_2)$.

\item If $0\leq z_1\leq z_2 \leq s$ for all $s\in\mathbb{R}$, we have that $\displaystyle {z_1^2\over 3}\leq {z_2^2\over 3}<X_1$.
Then
$$f(z_2)=z_2\left({{z_2^2}\over 3}-X_1\right)=-z_2\left(X_1-{{z_2^2}\over 3}\right)\leq -z_1\left(X_1-{{z_2^2}\over 3}\right)\leq -z_1\left(X_1-{{z_1^2}\over 3}\right)=f(z_1).$$

\item If $0\leq z_1 \leq t$ y $s<z_2$ for some $t\in\mathbb{R}$ and for all $s\in\mathbb{R}$, then $f(z_1)\leq 0\leq f(z_2)$.

\end{enumerate}

\vspace{1ex}

\noindent Therefore, $f$ is negative and decreasing on $[1,X_1^2]\cap\left\{z:\, 0\leq z\leq s \;\;\text{for some $s\in\mathbb{R}$}\right\},$
and $f$ is positive and increasing on $[1,X_1^2]\cap\left\{z:\, s\leq z \;\;\text{for all $s\in\mathbb{R}$}\right\}$. However, $f(z)$ does not have a  minimum on this interval because $f'(z)\neq 0$ for all 
$z\in [1,X_1^2]$.
\end{example}

\noindent Although that $K$ is an ordered field that extending the order of $\mathbb{R}$, it is not possible to show a Rolle's theorem since $K$ is not henselian. We refer to the reader \cite{rolle} for a 
characterization of whose fields which satisfies a Rolle's Theorem for polynomials.

\vspace{1ex}

\section{Local invertibility of analytic functions}
\noindent Let $X\subset K$ be a set with no isolated points. For a function
  $f:\, X\rightarrow K$ we consider
 $$\|f\|_{\infty}=\sup_{\Gamma^{\#}}\{v(f(x)):\; x\in X\}$$
where $\Gamma^{\#}$ denote the Dedekind completion of $\Gamma$. 

\vspace{1ex}
 
 \noindent As in the case of fields with a rank one valuation (\cite{Wim}), an analytic function  $f:X\rightarrow K$ is continuously differentiable on $X$, that is, the function
$$\Phi_1 f(x,y)= {{f(x)-f(y)}\over{x-y}}\quad\quad\quad\, (x,y\in X,\; x\neq y),$$
can be entended to a continuous function on $X\times X$.

\vspace{1ex}

\noindent The set
 $C^1(X\rightarrow K)$ consisting of  the continuously differentiable functions on $X$ is complete with the topology induced by the non-archimedean norm 
 $$\|f\|_1=\max\{\|f\|_{\infty},\|\Phi_1f\|_{\infty}\}.$$

\noindent Using the ultrametrizability of the uniformity induced by $\|\,\|_1$ on $C^1(X\rightarrow K)$, its possible to prove an implicit function theorem for functions $f:\, K^2\rightarrow K$ 
(for more details, see \cite{Hector2}). The proof is an application of Banach fixed point theorem on this space.

\vspace{1ex}

\noindent Since analytic functions are $C^1$- functions, we will extend this argument to prove that the local inverse $g$ of an analytic function $f$ with derivate non-null on an open set $U$ 
 is also an analytic function. For this purpose, we consider $\mathcal{A}(X\rightarrow K)$ the set of  analytics functions $f$ on $X$ with the uniform norm $\|\cdot\|_{\infty}$.
  
 \vspace{1ex}
 
 \noindent We have that $(\mathcal{A}(X\rightarrow K), \|\cdot\|_{\infty} )$ is a normed space over $K$ in the sense of \cite{morado} and the norm induced a topology. Moreover, as in the case of $(C^1(X\rightarrow K),\|\cdot\|_1)$, the norm $\|\cdot\|_{\infty}$ induces an uniformity $\mathcal{U}$ on $\mathcal{A}(X\rightarrow K)$, which has a base the collection of sets
$$S_r=\{(f,g):\; f,g\in \mathcal{A}(X\rightarrow K),\, \|f-g\|_{\infty}<r\}$$
for all $r\in\Gamma^{\#}$.

\vspace{1ex}

\noindent Following the ideas of \cite{Hector2}, the uniform space $(\mathcal{A}(X\rightarrow K),\mathcal{U})$ is (ultra)metrizable and the uniformity is induced by the ultrametric $d_{\infty}:\,\mathcal{A}(X\rightarrow K)\times \mathcal{A}(X\rightarrow K)\rightarrow \mathbb{R}^+$ defined as 
$$d_{\infty}(f,g):=\max_{x\in X}\{d(f(x),g(x))\}.$$
where $d: K \rightarrow  \mathbb{R}^+$ is the ultrametric defined on $K$ (see Preliminaries). 

\vspace{1ex}

\noindent We have that $d_{\infty}$ and $\|\, \|_{\infty}$ induce the same Cauchy sequences on $\mathcal{A}(X\rightarrow K)$, since 
$$\{z\in K:\; d(z,a)< \frac{1}{2^{n}}\}\subset \{z\in K:\; v(z-a)< \hat{g}_n^{-1}\}\subset\{z\in K:\; d(z,a)< \frac{1}{2^{n-1}}\}$$
for all $a\in K$ and $n\in\mathbb{N}\setminus\{0\}$.
Then $(\mathcal{A}(X\rightarrow K),\mathcal{U})$ is complete
if and only if each Cauchy sequence converges in $(\mathcal{A}(X\rightarrow K),\mathcal{U})$. 

\begin{thm}\label{analytic}
Let $a\in K$ and $r\in\Gamma$.  Let $p_n: B_a(r)\rightarrow K$ a Cauchy sequence of polynomial in $K[X]$ with respect to $\|\;\|_{\infty}$ in $B_a(r)$,
then $p_n$ converges uniformly to an analytic function $f:B_a(r)\rightarrow K$.
\end{thm}

\begin{proof}
Let $r\in\Gamma$, without loss of generality we suppose that $a=0$. Let $\epsilon>0$ with $\epsilon<r^i$ for all $i\in\mathbb{N}$ and $(p_n(z))_n$ is a Cauchy sequence of polynomials (in $K[x]$) with respect to $\|\;\|_{\infty}$ on $B_0(r)$. We consider the following cases:

\vspace{1ex}

\noindent If $\{deg(p_n(z)):\, n\in\mathbb{N}\}$ is finite, then the elements of $(p_n(z))_n$ are of the form
$$p_n(z)=\sum_{i=0}^{m}a_i^{(n)}z^i\quad\quad\left(a_i^{(n)}\in K\right)$$
for some $m\in\mathbb{N}$. Since $(p_n)$ is a Cauchy sequence, there exists $N\in \mathbb{N}$ such that $\|p_t-p_u\|_\infty<\epsilon^2$ for all $t,u\geq N$
and for all $i\in\{0,..,m\}$. Applying the Maximum Principle to $(p_t-p_u)$, we observe that
$$v\left(a_i^{(t)}-a_i^{(u)}\right)r^i\leq \max_{1\leq i\leq m}v\left(a_i^{(t)}-a_i^{(u)}\right)r^i=\|p_t-p_u\|_\infty<\epsilon^2$$
Therefore, by the choose of $\epsilon$, $v(a_i^{(t)}-a_i^{(u)})<\epsilon$ for all $t,u\geq N$ and for all $i\in\{0,..,m\}$, which implies that
 the sequence $(a_i^{(n)})_{n}$ is Cauchy uniformly in $i$ and hence converges in $K$. Let
$a_i=\displaystyle\lim_{n\rightarrow\infty} a_i^{(n)}$, then using classical arguments it can be prove directly that $(p_n)_n$ converges uniformly to 
$p(z)=\displaystyle\sum_{i=0}^{m}a_iz^i$ on $B_0(r)$.

\vspace{1ex}

\noindent If $\{deg(p_n(z)):\,n\in\mathbb{N}\}$ is infinite, we consider a subsequence $(q_n(z))_n$ of $(p_n(z))_n$ with $deg(q_k)<deg(q_{k+1})$, and clearly $(q_n(z))_n$ is a Cauchy sequence.
We suppose that
$$q_n(z)=\sum_{i=0}^{M_n}a_{i}^{(n)}z^{i}$$
where $M_n=deg(q_n(z))$. Since $a_i^{(n)}=0$ if $i\geq M_n$, we write
 $$q_n(z)=\sum_{i=0}^{\infty}a_{i}^{(n)}z^{i}.$$

\noindent Using similar arguments of the previous case, there exists $N\in\mathbb{N}$ such that
 $$v(a_i^{(t)}-a_i^{(u)})<\epsilon\quad\quad (t,u\geq N)\qquad\qquad (\ast)$$
for all $i\in\mathbb{N}$. Therefore the sequence $(a_i^{(n)})_{n}$ is Cauchy uniformly in $i$ and hence converges in $K$.
Let $a_i=\displaystyle\lim_{n\rightarrow\infty}a_i^{(n)}$. We fix $u$ and then $t\rightarrow\infty$ in $(\ast)$, we have that $v(a_i-a_i^{(N)})<\epsilon$
 for all $i\in\mathbb{N}$. Moreover, $a_{i}^{(N)}=0$ if $i>M_N$, which implies
$$v(a_i) = v(a_i-a_i^{(N)})<\epsilon\quad\quad (i>M_N).$$
 Therefore $\displaystyle\lim_{i\rightarrow\infty}a_i=0$, and by theorem \ref{convergencia series} the function $p(z)=\displaystyle\sum_{i=0}^{\infty}a_iz^i$
 is well defined on $B_0(r)$. Using classical arguments it can be prove directly that $(q_n(z))_n$ converges uniformly to $p(z)$ on $B_0(r)$.
On the another hand,
 $$\|p_n-p\|_{\infty}\leq \max\{\|p_n-q_n\|_{\infty}\|q_n-p\|_{\infty}\}.$$
 Since $(q_n(z))_n$ is a subsequence of $(p_n(z))_n$ and that is a Cauchy sequence uniformly on $B_0(r)$, then $(p_n(z)-q_n(z))_n\rightarrow 0$ if $n\rightarrow \infty$ uniformly in $B_0(r)$. Therefore, we can conclude that the sequence of polynomial $p_n(z)$ converges uniformly to an analytic function $p(z)$.
\end{proof}

\begin{cor}\label{analytic complete}
Let $(f_n)_n$ a sequence of analytic functions defined on $B_a(r)$ for some $r\in \Gamma$. If $\displaystyle\lim_{n\rightarrow\infty}f_n=f$ uniformly on $B_a(r)$, then $f$
is an analytic function on $B_a(r)$.
\end{cor}

\begin{proof}
Let $\epsilon\in\Gamma$. There exists a sequence of polynomials $(p_n)_n$ such that
$\|f_n-p_n\|_{\infty}<\epsilon$ for all $n\in\mathbb{N}$. For all
$$\|f-p_n\|_{\infty}\leq \max\{\|f-f_n\|_{\infty},\|f_n-p_n\|_{\infty}\}.$$
Since  $(f_n)_n$ converges we have that $\displaystyle\lim_{n\rightarrow\infty}p_n=f$ uniformly on $B_a(r)$, and by the previous theorem $f$ is analytic on $B_a(r)$.
\end{proof}

\begin{cor}\label{analytic metric}
$(\mathcal{A}(B_a(r)\rightarrow K),\mathcal{U})$ is a complete uniform space.
\end{cor}
\begin{proof}
Let $(f_n)_n$ be a Cauchy sequence in $(\mathcal{A}(B_a(r)\rightarrow K),\mathcal{U})$. 
By the previous assertions, we only must prove that $(f_n)_n$ converges in $(\mathcal{A}(B_a(r)\rightarrow K),\|\,\|_{\infty})$.

\vspace{1ex}

\noindent Let $\epsilon\in \Gamma^{\#}$, then there exists $m\in\mathbb{N}$ such that for all $k,n\geq m$, 
$$\|f_k-f_n\|_{\infty}<\epsilon.$$

\noindent On the another hand, for each $n\in \mathbb{N}$ there exists a polynomial $p_n\in K[x]$ 
such that $\|f_n-p_n\|_{\infty}<\epsilon$. We have that
$$\|p_n-p_k\|_{\infty}\leq \max\{\|f_n-p_n\|_{\infty}, \|f_n-f_k\|_{\infty},\|f_k-p_k\|_{\infty}\}<\epsilon,$$
for all $k,n\geq m$, that is, $(p_n)_n$ is also a Cauchy sequence in $(\mathcal{A}(B_a(r)\rightarrow K),d_{\infty})$. Then by theorem \ref{analytic}, $(p_n)$ converges uniformly to an analytic function $f$ on $B_a(r)$. But

$$\|f-f_n\|_{\infty}\leq\max\{\|f-p_n\|_{\infty},\|p_n-f_n\|_{\infty}\}$$
and by usual arguments, we conclude that $f_n$ converges uniformly to $f$ on $B_a(r)$.
\vspace{1ex}
\end{proof}

\begin{example}[A continuous function that is not uniformly approximable on $K$ by polynomials] We remark that $\hat{k}=\{x+B_0(1^-):\, x\in\mathbb{R}\}$. Let $f:B_0(\hat{g}_1)\rightarrow K$ defined as follows
\[ f(z)=\left\{
\begin{array}{ll}
 n & \text{if $z\in n+B_0(1^-)$}\\
 0 & \text{otherwise}
\end{array}
\right. 
\]
where $n\in\mathbb{N}$. Since each ball $B_x(1^-)$ and $B_0(\hat{g}_1)$ are clopen, $f$ is continuous on $B_0(\hat{g}_1)$. 

\vspace{1ex}

\noindent Let $\epsilon<\frac{1}{X_1}$. If there exists $p(z)\in K[x]$ such that $v(f(z)-p(z))<\epsilon$ for all $z\in B_0(\hat{g}_1)$ then $f(z)=p(z)+w(z)$ with $v(w(z))<\epsilon$. Applying theorem \ref{modulo maximo} to $p(z)$, $v(p(z))$ takes its maximum on $\{z\in K:\; v(z)=\hat{g}_1\}$, but on this set
\begin{align*}
v(p(z))& =v(f(z)+(p(z)-f(z)))\\
&\leq\max\{v(f(z)),v(p(z)-f(z))\}\\
&\leq\max\{0,v(w(z))\}<\epsilon<1.
\end{align*}
However, $v(p(n))=v(f(n)-w(n))=v(n-w(n))=1$ for all $n\in\mathbb{N}$ since $v(w(n))<\epsilon<1$.
\end{example}

\noindent Using the corollary \ref{analytic metric} and usual arguments, we can prove directly the following corollary.

\begin{cor}\label{completo2}
If $C$ is closed set in $K$, then $(\mathcal{A}(X\rightarrow C),d_{\infty})$ is a complete metric space.
\end{cor}

\noindent We recall that the proof of the following result is an extention of the proof of local invertibility theorem for $C^1$ functions and the implicit theorem proved in \cite{Hector1} and \cite{Hector3} respectively. The principal argument is the application of fixed point Banach theorem on a complete metric space.

\vspace{1ex}

\begin{thm}[{\bf Invertibility theorem for analytic functions}]
Let $f:B_{x_0}(r)\rightarrow K$ a non constant analytic function such that $f'(x_0)\neq 0$. Then, there exist a neighbourhood $U$ of $x_0$ such that f is injective on $U\cap B_{x_0}(r)$ and $g:f(U\cap B_{x_0}(r))\rightarrow U\cap B_{x_0}(r)$, the local inverse of $f$, is analytic on $f(U\cap B_{x_0}(r))$.
\end{thm}

\begin{proof}

Let us suppose that $f$ is analytic on $B_{x_0}(r)$ and let $s=f'(x_0)\neq 0$. We remark that $f$ is also a $C^1$ function on $X$, then there is $r_1\in\Gamma$ such that
$$v\left(\frac{f(x)-f(y)}{x-y}-f'(x_0)\right)<\hat{g}_1^{-1}v(f'(x_0)).$$
for each $x,y\in B_{x_0}(r_1)$ with $x\neq y$. Using the strong triangle inequality we can conclude that
$$v(f(x)-f(y))=v(s)v(x-y)\,\,\; (x,y\in  B_{x_0}(r_1)).$$

\noindent On the another hand, the function  $g(x,y)=f(x)-y$ is continuos in $(x_0,y_0)$ and $g(x_0,y_0)=0$, then there exists $\delta_1\in\Gamma$ such that $v(f(x_0)-y)\leq v(s)\cdot r_1$ for each $y\in B_{x_0}(\delta_1)$. 

\vspace{1ex}

\noindent We choose $\delta=\min\{r_1,\delta_1\}$ and  consider the set
$\mathcal{A}(B_{y_0}(\delta)\rightarrow B_{x_0}(r_1))$ with the metric $d_{\infty}$. The set $B_{x_0}(r_1)$ is closed on $K$, then by Corollary \ref{completo2}, $(\mathcal{A}(B_{y_0}(\delta)\rightarrow B_{x_0}(r_1)),d_{\infty})$ is a complete metric space.

\vspace{1ex}

\noindent For $\psi\in \mathcal{A}(B_{y_0}(\delta)\rightarrow B_{x_0}(r_1))$, we define the function $$h:\mathcal{A}(B_{y_0}(\delta)\rightarrow B_{x_0}(r_1))\rightarrow \mathcal{A}(B_{y_0}(\delta)\rightarrow B_{x_0}(r_1))$$ as follows
$$(h(\psi))(y):=\psi(y)-s^{-1}(f(\psi(y))-y).$$
 
\noindent We claim that $h$ is well defined. Indeed, let $y\in B_{y_0}(\delta)$ and $\psi\in (\mathcal{A}(B_{y_0}(\delta)\rightarrow B_{x_0}(r_1)),d_{\infty})$. Then $\psi(y)\in B_{x_0}(r_1)$ and
\begin{align*}
v((h(\psi))(y)-x_0)&=v(\psi(y)-s^{-1}(f(\psi(y))-y)-x_0)\\
&\leq\max\{v(\psi(y)-x_0),s^{-1}(f(\psi(y))-f(x_0)),s^{-1}(f(x_0)-y)\}.
\end{align*}
Since $\psi(y)\in B_{x_0}(r_1)$ and $r_1\leq r$, from the first observation 
$$v(f(\psi(y))-f(x_0))=v(s)v(\psi(y)-x_0).$$

\noindent On the other hand, by the choose of $\delta$ and the continuity of $g(x,y)$ in $(x_0,y_0)$, it follows that $v(f(x_0)-y)) \leq r_1$ if $v(y-y_0)\leq \delta$. Hence, $v((h(f))(x)-y_0)\leq r_1$ and $(h(\psi))(y)\in B_{x_0}(r_1)$. By the Proposition \ref{composicion de analiticas} we have that $f\circ \psi$ is analytic on $B_{y_0}(\delta)$, we conclude that $h(\psi)\in \mathcal{A}(B_{y_0}(\delta)\rightarrow B_{x_0}(r_1))$ and it is well defined.

\vspace{1ex}

\noindent For finishing the proof, we show that $h$ is a contraction on $\mathcal{A}(B_{y_0}(\delta)\rightarrow B_{x_0}(r_1))$.

\vspace{1ex}

\noindent Let $\psi,\varphi \in \mathcal{A}(B_{y_0}(\delta)\rightarrow B_{x_0}(r_1))$, then for all $y\in B_{y_0}(\delta)$ with $\psi(y)\neq \varphi(y)$
\begin{align*}
v((h(\psi))(y)-(h(\varphi))(y)) &=v((\psi(y)-\varphi(y))-s^{-1}(f(\psi(y))-f(\varphi(y)))\\
&=v(s^{-1}(\psi(y)-\varphi(y))\; v\left(s-\frac{f(\psi(y))-f(\varphi(y))}{\psi(y)-\varphi(y)}\right).
\end{align*}

\noindent But,  we have that $\psi(y),\varphi(y)\in B_{x_0}(r_1)$ for all $y\in B_{y_0}(\delta)$, and
$$\sup_{\Gamma^{\#}}\left\{v\left(\frac{f(x)-f(y)}{x-y}-f'(x_0)\right):\,x,y\in B_{x_0}(r_1),\, x\neq y\right\}<\hat{g}_1^{-1}v(f'(x_0)).$$
which implies that
$$v\left(s-\frac{f(\psi(y))-f(\varphi(y))}{\psi(y)-\varphi(y)}\right)\leq \hat{g}_1^{-1}v(s).$$
Therefore $v((h(\psi))(y)-(h(\varphi))(y))< \hat{g}_1^{-1} v(\psi(y)-\varphi(y)).$ Since the order and the valuation are compatible on $K$, we have that $$|(h(\psi))(y)-(h(\varphi))(y)|<|X_1^{-1}(\psi(y)-\varphi(y))|.$$ 
But $\phi$ is an increasing function and satisfies the following inequality (see \cite{Hector1} for more details):
$$\phi(|(h(\psi))(y)-(h(\varphi))(y)|)\leq \phi(|X_1^{-1}(\psi(y)-\varphi(y))|)\leq \phi(|X_1^{-1}|)\phi(|\psi(y)-\varphi(y)|).$$
 Therefore, by the definition of $d$, we have that
$$d((h(\psi))(x),(h(\varphi))(x))\leq\frac{1}{2}d(\psi(x),\varphi(x)))\qquad (x\in B_{x_0}(\delta), \psi(x)\neq \varphi(x)).$$
If $\psi(x)=\varphi(x)$ for some $x\in B_{x_0}(r)$, then $(\psi(f))(x)=(\varphi(g))(x)$ and the inequality above is true on $B_{x_0}(r)$. Hence, for all $x\in B_{x_0}(r)$
$$d((h(\psi))(x),(h(\varphi))(x))\leq\frac{1}{2}d(\psi(x),\varphi(x)))\leq \frac{1}{2}d_{\infty}(\psi,\varphi),$$
which implies that $h$ is a contraction, and applying the Banach fixed point theorem  there exists a unique function $$w(z)\in \mathcal{A}(B_{y_0}(\delta)\rightarrow B_{x_0}(r_1))$$
 such that $(h(w)(y)=w(y)$. But $(h(w))(y)=w(y)-s^{-1}((f\circ\, w)(y)-y)$, then $(f\circ\, w)(y)=y$ for all $y\in B_{y_0}(\delta)$. 

\end{proof}

\begin{cor}
Let $f:\,D\rightarrow K$ a non-constant analytic function such that $f'(z)\neq 0$ on $D$. If $U\subset D$ is an open set then $f(U)$ is open. 
\end{cor}

\end{document}